\documentclass[a4paper, 11pt]{amsart}   
\usepackage{mathptmx, amssymb,amscd,latexsym, eulervm}   
\usepackage{amsmath}
\usepackage{amsthm}
\usepackage{colonequals}
\usepackage{mathdots}
\usepackage{color}
\definecolor{chianti}{rgb}{0.6,0,0}
\definecolor{meretale}{rgb}{0,0,.6}
\definecolor{leaf}{rgb}{0,.35,0}
\usepackage[colorlinks=true, pagebackref, hyperindex, citecolor=meretale, urlcolor=leaf, linkcolor=chianti]{hyperref}
\usepackage{color}
\usepackage[onehalfspacing]{setspace}
\usepackage{tabularx}
\usepackage{amsfonts}
\usepackage{paralist}
\usepackage{aliascnt}
\usepackage[initials, lite]{amsrefs}
\usepackage{amscd}
\usepackage{blkarray}
\usepackage{mathbbol}
\usepackage{setspace}
\usepackage[inner=2.5cm,outer=2.5cm, bottom=3.2cm]{geometry}
\usepackage{tikz, tikz-cd}
\usepackage{calligra,mathrsfs}

\usepackage{tikz}
\usetikzlibrary{matrix}
\usetikzlibrary{arrows,calc}
\allowdisplaybreaks

\BibSpec{collection.article}{%
	+{}  {\PrintAuthors}                {author}
	+{,} { \textit}                     {title}
	+{.} { }                            {part}
	+{:} { \textit}                     {subtitle}
	+{,} { \PrintContributions}         {contribution}
	+{,} { \PrintConference}            {conference}
	+{}  {\PrintBook}                   {book}
	+{,} { }                            {booktitle}
	+{,} { }                            {series}
	+{, vol.} { }                            {volume}
	+{,} { }                            {publisher}
	+{,} { \PrintDateB}                 {date}
	+{,} { pp.~}                        {pages}
	+{,} { }                            {status}
	+{,} { \PrintDOI}                   {doi}
	+{,} { available at \eprint}        {eprint}
	+{}  { \parenthesize}               {language}
	+{}  { \PrintTranslation}           {translation}
	+{;} { \PrintReprint}               {reprint}
	+{.} { }                            {note}
	+{.} {}                             {transition}
	+{}  {\SentenceSpace \PrintReviews} {review}
}
\AtBeginDocument{%
	\def\MR#1{}
}

\makeatletter
\@namedef{subjclassname@2020}{%
	\textup{2020} Mathematics Subject Classification}
\makeatother

\newcommand{\kk}{\mathbb{k}}

\newcommand{\pp}{\mathfrak{p}}
\newcommand{\aaa}{\mathfrak{a}}
\newcommand{\ccc}{\mathfrak{c}}
\newcommand{\bb}{\mathfrak{b}}

\newcommand{\Quot}{{\normalfont\text{Quot}}}

\newcommand{\Diff}{{\rm Diff}}
\newcommand{\IM}{\normalfont\text{Im}}

\newcommand{\Ass}{\normalfont\text{Ass}}

\newcommand{\Hom}{{\normalfont\text{Hom}}}

\newcommand{\Rred}{{R_{\rm red}}}

\newcommand{\Spec}{{\normalfont\text{Spec}}}

\newtheorem{theorem}{Theorem}[section]

\newtheorem{headthm}{Theorem}

\newaliascnt{headcor}{headthm}
\newtheorem{headcor}[headcor]{Corollary}
\aliascntresetthe{headcor}

\newaliascnt{headconj}{headthm}

\aliascntresetthe{headconj}

\newaliascnt{corollary}{theorem}
\newtheorem{corollary}[corollary]{Corollary}
\aliascntresetthe{corollary}

\newaliascnt{claim}{theorem}

\aliascntresetthe{claim}

\newaliascnt{lemma}{theorem}
\newtheorem{lemma}[lemma]{Lemma}
\aliascntresetthe{lemma}

\newaliascnt{conjecture}{theorem}

\aliascntresetthe{conjecture}

\newaliascnt{proposition}{theorem}
\newtheorem{proposition}[proposition]{Proposition}
\aliascntresetthe{proposition}

\theoremstyle{definition}
\newaliascnt{definition}{theorem}
\newtheorem{definition}[definition]{Definition}
\aliascntresetthe{definition}

\newaliascnt{notation}{theorem}

\aliascntresetthe{notation}

\newaliascnt{example}{theorem}

\aliascntresetthe{example}

\newaliascnt{examples}{theorem}

\aliascntresetthe{examples}

\newaliascnt{remark}{theorem}
\newtheorem{remark}[remark]{Remark}
\aliascntresetthe{remark}

\newaliascnt{question}{theorem}

\aliascntresetthe{question}

\newaliascnt{questions}{theorem}

\aliascntresetthe{questions}

\newaliascnt{problem}{theorem}

\aliascntresetthe{problem}

\newaliascnt{construction}{theorem}

\aliascntresetthe{construction}

\newaliascnt{setup}{theorem}
\newtheorem{setup}[setup]{Setup}
\aliascntresetthe{setup}

\newaliascnt{algorithm}{theorem}

\aliascntresetthe{algorithm}

\newaliascnt{observation}{theorem}

\aliascntresetthe{observation}

\newaliascnt{defprop}{theorem}

\aliascntresetthe{defprop}

\DeclareFontFamily{OT1}{pzc}{}
\DeclareFontShape{OT1}{pzc}{m}{it}{<-> s * [1.100] pzcmi7t}{}
\DeclareMathAlphabet{\mathchanc}{OT1}{pzc}{m}{it}

\def\equationautorefname~#1\null{(#1)\null}
\def\sectionautorefname~#1\null{Section #1\null}
\def\subsectionautorefname~#1\null{\S #1\null}

\begin{document}
	
	% ---------------Tittle and presentation----------------------------
	\title{Uniformity in nonreduced rings via Noetherian operators}

	\author{Yairon Cid-Ruiz}
	\address{Department of Mathematics, North Carolina State University, Raleigh, NC 27695, USA}
	\email{ycidrui@ncsu.edu}
	
	\author{Jack Jeffries}
	\address{Department of Mathematics, University of Nebraska-Lincoln, 203 Avery Hall, Lincoln, NE 68588, USA}
	\email{jack.jeffries@unl.edu}
	
%	\thanks{Both authors were partially supported by NSF grant DMS-1928930 and  Alfred P. Sloan Foundation G-2021-16778 while in residence at SLMath. The second author was partially supported by NSF CAREER Award DMS-2044833.}

\date{\today}

	\maketitle
	
\begin{abstract}
	We prove a differential version of the Artin-Rees lemma with the use of Noetherian differential operators.
	As a consequence, we obtain several uniformity results for nonreduced rings.
\end{abstract}

\section{Introduction}

A curious historical tradition is that analytic and algebraic techniques have often competed to prove equivalent results. 
This paper lies at the crossroads of the following two topics where both algebraic and analytic techniques have played essential roles: 

\begin{enumerate}[(a)]
	\item\label{topic_a} \emph{Primary ideals and Noetherian operators:}  
	It is known that one can use Macaulay's inverse systems to describe zero-dimensional primary ideals in terms of differential operators (see \cite{GROBNER_LIEGE, GROBNER_MATH_ANN, MACAULAY}). 
	However, analysts first described arbitrary primary ideals via differential operators in polynomial rings over the complex numbers.
	Indeed, a main step in the Fundamental Principle of Ehrenpreis and Palamodov \cite{EHRENPREIS, PALAMODOV} is the characterization of primary ideals by differential operators.
	Following Palamodov's notation, a set of differential operators describing a primary ideal is called a set of \emph{Noetherian (differential) operators}.
	Subsequent important algebraic developments were made by Brumfiel \cite{BRUMFIEL_DIFF_PRIM} and Oberst \cite{OBERST_NOETH_OPS}.
	More recently, this theme has seen a rebirth in the areas of commutative algebra and computational algebra (see \cite{NOETH_OPS, PRIM_DIFF_DEC, PRIM_DIFF_EQ, M2_NOETH_OPS,LINEAR_PDE_STURMFELS, CHEN_LEYKIN}).
	
	\smallskip
	\item \label{topic_b} \emph{Integral closure and the Brian\c{c}on-Skoda theorem:} The Brian\c{c}on-Skoda theorem \cite{BRIANCON_SKODA}, proved initially in the complex analytic case, states that given an ideal $I$, there is a constant $c$ such that $\overline{I^{n+c}} \subset I^n$ for any $n \ge 0$, where $\overline{I^{n+c}}$ denotes the integral closure of $I^{n+c}$.
	The first algebraic proof of this result was given by Lipman and Sathaye for regular rings \cite{LIPMAN_SATHAYE} (also, see \cite{HH_TIGHT_CLOSURE,huneke2006integral}).
	Then, Huneke went even further in his uniform bounds paper \cite{HUNEKE_UNIFORM}, where he showed that for large classes of reduced rings one could choose a uniform constant $c$ that is valid for all ideals $I$. The reducedness hypothesis is necessary since the nilradical is contained in the integral closure of any ideal.
\end{enumerate}

\noindent
Our point of departure is a paper of Sznajdman \cite{Sznajdman} that combines \hyperref[topic_a]{(a)} and \hyperref[topic_b]{(b)} in a spectacular way: in a complex analytic setting, it is shown that one can use Noetherian operators to provide a version of the Brian\c{c}on-Skoda theorm for nonreduced analytic spaces. 
Following historical tradition, we extend Sznajdman's result to nonrestrictive algebraic settings. 

\smallskip

Our main result is a differential version of the Artin-Rees lemma with Noetherian operators.
As a consequence, we obtain a metatheorem stating that \emph{``uniform containment type results hold for many classes of rings when they hold for the corresponding reduced ring.''}

\smallskip

Let $R$ be an algebra essentially of finite type over a field $\kk$ and assume that $R$ has no embedded associated primes. 
Let $\Rred = R / \sqrt{0}$ be the reduced ring corresponding to $R$.
For any ideal $\aaa \subset R$ such that $\Ass_R(R/\aaa) \subseteq \Ass(R)$, one can find a finite set of differential operators $\delta_1,\ldots,\delta_m \in \Diff_{R/\kk}(R, \Rred)$ such that 
$$
\aaa \;=\; \lbrace f \in R \mid \delta_i(f) = 0 \text{ for all } 1 \le i \le m \rbrace;
$$
here we say that $\delta_1,\ldots,\delta_m$ is a \emph{set of Noetherian operators describing the ideal $\aaa$} (see \autoref{cor_noeth_ops}); when $\aaa=(0)$, we call such a set a \emph{set of Noetherian operators for $R$}.
Notice that a set of Noetherian operators $\delta_1,\ldots,\delta_m \in \Diff_{R/\kk}(R, \Rred)$ for $R$ yields a $\kk$-linear inclusion $R \hookrightarrow \left(\Rred\right)^m$.
Given an ideal $I \subset \Rred$ and a finite set of differential operators $\delta_1,\ldots,\delta_m \in \Diff_{R/\kk}(R, \Rred)$, we set the following notation
$$
I :_{\Rred} \lbrace\delta_1,\ldots,\delta_m\rbrace \;\colonequals\; \lbrace f \in R \mid \delta_i(f) \in I \text{ for all } 1 \le i \le m \rbrace \;=\; \bigcap_{i=1}^m \delta_i^{-1}(I). 
$$

We now present our main result.

\begin{headthm}[Differential Artin-Rees]
	\label{thm:main}
	Let $R$ be an algebra essentially of finite type over a field $\kk$ and assume that $R$ has no embedded associated primes. 
	Let $\Rred = R / \sqrt{0}$.
	Then, there exists a set of  Noetherian operators $\delta_1,\ldots,\delta_m \in \Diff_{R/\kk}(R, \Rred)$ for $R$ and an integer $c\ge0$ such that 
	$$
	I^{n+c} :_{\Rred} \lbrace\delta_1,\ldots,\delta_m\rbrace \;\subseteq\; J^n
	$$ 
	for any ideal $J$ in $R$ with image $I$ in $\Rred$ and any $n \ge 0$.
\end{headthm}	

Complementary to \autoref{thm:main}, we also have a quite basic reverse inclusion (see \autoref{rem_other_containment}).
The main corollary is the algebraic analogue of the result of Sznajdman \cite[Theorem~1.2]{Sznajdman}.

\begin{headcor}[Brian\c{c}on-Skoda theorem for nonreduced rings]
	\label{cor:Briancon_Skoda}
	Under the assumptions and notation of \autoref{thm:main}, there exists a set of  Noetherian operators $\delta_1,\ldots,\delta_m \in \Diff_{R/\kk}(R, \Rred)$ for $R$ and an integer $c\ge0$ such that 
	$$
	\overline{I^{n+c}} :_{\Rred} \lbrace\delta_1,\ldots,\delta_m\rbrace \;\subseteq\; J^n
	$$ 
	for any ideal $J$ in $R$ with image $I$ in $\Rred$ and any $n \ge 0$.
\end{headcor}

Our second application is to establish an analogue of a celebrated containment statement of symbolic powers for not necessarily regular rings in the case $\Rred$ is regular (see \cite{HH_SYMB_POWERS, EIN_LAZ_SMITH,MS_SYMB}). 

\begin{headcor}[Containment of symbolic powers for nonregular rings]
	\label{cor:symb_powers}
	Keep the assumptions and notation of \autoref{thm:main}.
	Moreover, assume that $\Rred$ is a regular ring of finite dimension $d$.
	Then, there exists a set of  Noetherian operators $\delta_1,\ldots,\delta_m \in \Diff_{R/\kk}(R, \Rred)$ for $R$ and an integer $c\ge0$ such that 
	$$
	I^{(nd+c)} :_{\Rred} \lbrace\delta_1,\ldots,\delta_m\rbrace \;\subseteq\; J^n
	$$ 
	for any ideal $J$ in $R$ with image $I$ in $\Rred$ and any $n \ge 0$.
\end{headcor}

\section{Proofs of our results}	

In this section, we provide the proofs for all our results.
We start with the following basic lemma showing the existence of certain special filtrations we need.

	\begin{lemma}
	\label{lem_filtration}
	Let $R$ be a Noetherian ring.
	Then, there is a filtration of $R$ by ideals 
	$$
	0 = \aaa_0 \,\subsetneqq\, \aaa_1 \,\subsetneqq\, \cdots \,\subsetneqq\, \aaa_k = R
	$$
	satisfying the following: 
	\begin{samepage}
	\begin{enumerate}[\rm (i)]
		\item\label{part-1} $\aaa_i/\aaa_{i-1}$ is isomorphic to a submodule of $R/\pp$ for some $\pp \in \Ass(R)$.
		\item\label{part-2}  $\Ass_R(R/\aaa_{i-1}) \subseteq \Ass(R)$ for $1\le i \le k$.
		\item\label{part-3} $\aaa_{k-1} = \pp$ for some $\pp \in \Ass(R)$.
	\end{enumerate}
	\end{samepage}
\end{lemma}
\begin{proof}
	Let $\pp_0 \in \Ass(R)$.
	By \cite[Lemma 3.5.3]{FLENNER_O_CARROLL_VOGEL},  there exists a filtration ${0 = \aaa_0 \,\subsetneqq\, \aaa_1 \,\subsetneqq\, \cdots \,\subsetneqq\, \aaa_{k-1} = \pp_0}$ such that each $\aaa_i/\aaa_{i-1}$ is isomorphic to an ideal of $R/\pp$ for some associated prime $\pp$ of $\pp_0$ considered as a submodule of $R$; setting $\aaa_k = R$ then yields a filtration evidently satisfying conditions \hyperref[part-1]{(i)} and \hyperref[part-3]{(iii)}.
	Then condition \hyperref[part-2]{(ii)} holds because $R/\aaa_{i-1}$ has a filtration with successive quotients isomorphic to submodules of $R/\pp$ for some $\pp \in \Ass(R)$.
\end{proof}

We now recall the notion of \emph{differential operators}. 
A general reference in this topic is \cite[\S 16]{EGAIV_IV}.

\begin{definition}
	Let $\kk$ be a field and $R$ be a $\kk$-algebra.
	Let $M, N$ be $R$-modules.
	The $n$-th order $\kk$-linear differential operators $\Diff_{R/\kk}^n(M, N) \subseteq \Hom_\kk(M, N)$ from $M$ to $N$ are defined inductively by:
	\begin{enumerate}[(i)]
		\item $\Diff_{R/\kk}^{0}(M,N) \colonequals \Hom_R(M, N)$.
		\item $\Diff_{R/\kk}^{n}(M, N) \colonequals \big\lbrace \delta \in \Hom_\kk(M,N) \mid [\delta, f] \in \Diff_{R/\kk}^{n-1}(M, N) \text{ for all } f \in R \big\rbrace$.
	\end{enumerate}
	Here we use the bracket notation $[\delta, f](m) \colonequals \delta(fm) - f\delta(m)$ for all $\delta \in \Hom_\kk(M, N)$, $f \in R$, and $m \in M$.
	The $\kk$-linear differential operators from $M$ to $N$ are given by 
	$$
	\Diff_{R/\kk}(M, N) \colonequals \bigcup_{n=0}^\infty \Diff_{R/\kk}^n(M,N).
	$$
\end{definition}

For completeness, we include a lemma that describes the behavior of differential operators under localization. 

\begin{lemma}\label{localization_lemma}
Let $R$ be an algebra essentially of finite type over a field $\kk$, $W\subseteq R$ be a multiplicatively closed subset, and $M,N$ be $R$-modules, with $M$ finitely generated. 
Then
\begin{enumerate}[\rm (i)]
\item\label{diff-part-1} $W^{-1} R \otimes_R \Diff_{R/\kk}(M,N) \cong \Diff_{W^{-1}R/\kk}(W^{-1}M,W^{-1}N)$, and
\item\label{diff-part-2} $\Diff_{R/\kk}(M,W^{-1}N) \cong \Diff_{W^{-1}R/\kk}(W^{-1}M,W^{-1}N)$.
\end{enumerate}
\end{lemma}
\begin{proof} Let $P^n_{R/\kk}$ denote the module of principal parts of order $n$ of $R$ over $\kk$, cf.,~ \cite[\S 16.3]{EGAIV_IV}. 
By \cite[Theorem~16.4.14]{EGAIV_IV} (cf., \cite[Proposition~2.16]{Quantifying}), for each $n$, we have $W^{-1} P^n_{R/\kk}\cong P^n_{W^{-1}R/\kk}$, so \[W^{-1}(M\otimes_R P^n_{R/\kk}) \cong W^{-1}M\otimes_{W^{-1}R} P^n_{W^{-1}R/\kk}.\]
Then, by the universal property of localization, we have
\begin{samepage}
\begin{align*} 
	\Diff_{R/\kk}^n(M,W^{-1}N) &\cong \Hom_R(M\otimes_R P^n_{R/\kk},W^{-1}N) \\&\cong \Hom_{W^{-1}R}(W^{-1}M\otimes_{W^{-1}R} P^n_{W^{-1}R/\kk},W^{-1}N )\cong \Diff_{W^{-1}R/\kk}^n(W^{-1}M,W^{-1}N),
 \end{align*}
 \end{samepage}
and since $P^n_{R/\kk}$ is finitely generated, we also have
\begin{samepage}\begin{align*}  
		W^{-1} R \otimes_R &\Diff_{R/\kk}^n(M,N) \cong W^{-1} \Hom_R(M\otimes_R P^n_{R/\kk},N)  \\ &\cong \Hom_{W^{-1}R}(W^{-1}M\otimes_{W^{-1}R} P^n_{W^{-1}R/\kk},W^{-1}N)
\cong \Diff_{W^{-1}R/\kk}^n(W^{-1}M,W^{-1}N). \qedhere\end{align*}
\end{samepage}
\end{proof}

We also use the following lemma, which follows along similar lines to \cite[Proposition~1.3]{BRUMFIEL_DIFF_PRIM}.

\begin{lemma}
	\label{order-lemma} 
	Let $R$ be a $\kk$-algebra, $M$, $N$ be $R$-modules, $J$ be an ideal of $R$, and $\delta$ be a differential operator from $M$ to $N$ of order at most $n$. 
	Then for any integer $t\geq 0$, $\delta(J^{n+t} M) \subseteq J^t N$.
\end{lemma}

The following theorem allows us to describe primary ideals with differential operators.

\begin{theorem}[{\cite{NOETH_OPS}}]
	\label{thm_noeth_ops}
	Let $R$ be an algebra essentially of finite type over a field $\kk$, $\pp \in \Spec(R)$ a prime ideal and $Q$ be a $\pp$-primary ideal.
	Then, there exists a finite set of differential operators $\delta_1,\ldots,\delta_m \in \Diff_{R/\kk}(R, R/\pp)$ such that 
	$$
	Q \;=\; \lbrace f \in R \mid \delta_i(f) = 0 \;\text{ for all \;$0 \le i \le m$}\rbrace.
	$$	
\end{theorem}

	A set of operators satisfying the conclusion of \autoref{thm_noeth_ops} is called a set of \emph{Noetherian operators describing $Q$}.
	We shall use the following setup for the rest of this section. 
	
\begin{setup}
	\label{setup}
	Let $R$ be an algebra essentially of finite type over a field $\kk$ and assume that $R$ has no embedded associated primes. 
	Let $\Rred \colonequals R / \sqrt{0}$ be the reduced ring corresponding to $R$.
	Let $\pi : R \rightarrow \Rred$ be the natural projection.
\end{setup}

A standard application of \autoref{thm_noeth_ops}	yields the possibility of describing certain (not necessarily primary) ideals in terms of differential operators. 	
For arbitrary ideals, we can use the notion of \emph{differential primary decomposition} from \cite{PRIM_DIFF_DEC}. 
	
\begin{corollary}
	\label{cor_noeth_ops}
	Assume \autoref{setup}.
	Let $\aaa \subset R$ be an ideal such that $\Ass_R(R/\aaa) \subseteq \Ass(R)$.
	There exists a set of Noetherian operators $\delta_1,\ldots,\delta_m \in \Diff_{R/\kk}(R, \Rred)$ describing the ideal $\aaa$; that is, 
	$$
	\aaa = \lbrace f \in R \mid \delta_i(f)=0 \;\text{ for all }\; 1 \le i \le m\rbrace.
	$$
\end{corollary}
\begin{proof}
Consider the multiplicatively closed set $W\colonequals R \setminus \bigcup_{\pp \in \Ass(R)}\pp$.
Let $S\colonequals W^{-1}R$.
By \autoref{localization_lemma}, we have the following isomorphisms 
\begin{align*}
	W^{-1}R \otimes_R \Diff_{R/\kk}(R, \Rred) &\;\cong\; \Diff_{S/\kk}(S, \Quot(\Rred))  \\
	&\;\cong\; \bigoplus_{\pp \in \Ass(R)} \Diff_{S/\kk}(R, R_\pp/\pp R_\pp)  \\
	&\;\cong\; \bigoplus_{\pp \in \Ass(R)} \Diff_{R_\pp/\kk}(R_\pp, R_\pp/\pp R_\pp).  
\end{align*}
Therefore, for any associated prime $\pp \in \Ass(R)$ and any differential operator $\delta' \in \Diff_{R_\pp/\kk}(R_\pp, R_\pp/\pp R_\pp)$, we can write $\delta' = \frac{1}{w} \cdot \delta$ for some $\delta \in \Diff_{R/\kk}(R, \Rred)$ and $w \in W$.

 Write an irredundant primary decomposition $\aaa = Q_1 \cap \cdots \cap Q_\ell$ where $Q_i$ is a $\pp_i$-primary ideal for some $\pp_i \in \Ass(R)$.
	From \autoref{thm_noeth_ops}, we obtain a set of Noetherian operators $$
	\delta_{i,1}',\ldots,\delta_{i, m_i}' \;\in\; \Diff_{R_{\pp_i}/\kk}(R_{\pp_i}, R_{\pp_i}/\pp_i R_{\pp_i})
	$$ 
	describing the $\pp_iR_{\pp_i}$-primary ideal $Q_iR_{\pp_i}$.
	We can write $\delta_{i,j}' = \frac{1}{w_{i,j}}\cdot \delta_{i,j}$  with $\delta_{i,j} \in \Diff_{R/\kk}(R, \Rred)$ and $w_{i,j} \in W$.
	Since $w_{i,j}$ is a nonzerodivisor on $\Rred$, we have
	\[ f\in Q_i \quad \text{if and only if} \quad \delta'_{i,j}(f) = 0 \ \text{for all $j$} \quad \text{if and only if} \quad \delta_{i,j}(f) = w_{i,j} \delta'_{i,j}(r) = 0 \ \text{for all $j$};\]
	that is, $\delta_{i,1},\ldots,\delta_{i, m_i} \in \Diff_{R/\kk}(R, \Rred)$ is a set of Noetherian operators describing the ideal $Q_i$.
	Finally, the union of all the $\delta_{i,j}$'s gives a set of Noetherian operators describing the ideal $\aaa$.
\end{proof}

The following technical proposition is our fundamental tool for obtaining the differential operators used in \autoref{thm:main}.

\begin{proposition}
	\label{prop_find_diff_op}
	Assume \autoref{setup}.
	Let $\aaa \subsetneqq \bb \subset R$ be two ideal ideals such that $\Ass_R(R/\aaa) \subseteq \Ass(R)$.
	Let $\psi : \bb/\aaa \hookrightarrow R/\pp$ be an injective $R$-linear map for some $\pp \in \Ass(R)$.
	Then, there is a differential operator $\delta_{\bb/\aaa} \in \Diff_{R/\kk}(R, \Rred)$ satisfying the following: 
	\begin{samepage}
	\begin{enumerate}[\rm (i)]
		\item $\delta_{\bb/\aaa}(\aaa) = 0$.
		\item Denoting by $\overline{\delta}_{\bb/\aaa} \in \Diff_{R/\kk}(R, R/\pp)$ the induced differential operator, we obtain a nonzero element ${d_{\bb/\aaa}} \in \Quot(R/\pp)$  such that 
		$$
		{\overline{\delta}}_{\bb/\aaa}(f) \;=\; \psi\left(\overline{f}\right) \cdot {d_{\bb/\aaa}}  \quad \text{ for all } \quad  f \in \bb.
		$$
	\end{enumerate}
	\end{samepage}
\end{proposition}	
\begin{proof}
	From \autoref{cor_noeth_ops}, we can choose Noetherian operators $\delta_1,\ldots,\delta_m \in \Diff_{R/\kk}(R, \Rred)$ that describe the  ideal $\aaa$. In particular, for any $f \in R$, we have that $f \in \aaa$ if and only if $\delta_i(f) = 0$ for all $1 \le i \le m$.
%	Let $\mathcal{E} \subset \Diff_{R/\kk}(R, \Rred)$ be the $(R \otimes_\kk R)$-submodule generated by $\delta_1,\ldots,\delta_m$.
%	Notice that, for any $f \in R$, we get $f \in \aaa$ if and only if $\delta(f) = 0$ for all $\delta \in \mathcal{E}$.
	
	Since $\bb \supsetneqq \aaa$, there exists a differential operator $\delta_{\bb/\aaa} \in \Diff_{R/\kk}(R, \Rred)$ of minimal order such that $\delta_{\bb/\aaa}(\bb) \neq 0$ and $\delta_{\bb/\aaa}(\aaa) = 0$.
	For all $f \in R$ and $g \in \bb$, we have the equality 
	$$
	\delta_{\bb/\aaa}(fg) \;=\; \pi(f) \cdot \delta_{\bb/\aaa}(g) + \left[\delta_{\bb/\aaa}, f\right](g).
	$$
	As $[\delta_{\bb/\aaa}, f]$ has strictly lower order than $\delta_{\bb/\aaa}$ and $[\delta_{\bb/\aaa}, f](\aaa)=0$, the minimality of $\delta_{\bb/\aaa}$ forces the vanishing $[\delta_{\bb/\aaa}, f](g)=0$.
	Thus we get $\delta_{\bb/\aaa}(fg) = \pi(f) \cdot \delta_{\bb/\aaa}(g)$ for all $f \in R$ and $g \in \bb$.
	This means that the restriction
	$$
	\delta_{\bb/\aaa} : \bb \rightarrow \Rred \quad \text{ is a nonzero $R$-linear map with } \quad \delta_{\bb/\aaa}(\aaa) = 0.
	$$
	Since $\pp \cdot \bb \subset \aaa$, we can make the identification $\delta_{\bb/\aaa} : \bb /\aaa \rightarrow (0:_{\Rred}\pp)$.
	Let $\tau : \Rred \rightarrow R/\pp$ be the quotient map.
	From the fact that $(0:_{\Rred} \pp) \cap \pp\Rred = 0$, we obtain a nonzero $R$-linear map
	$$
	\overline{\delta}_{\bb/\aaa} = \tau \circ \delta_{\bb/\aaa} : \bb/\aaa \rightarrow R/\pp.
	$$
	Let $\ccc \colonequals \IM(\psi) \subset R/\pp$ and notice that we have the isomorphisms 
	$$
	\Hom_{R/\aaa}(\bb/\aaa, R/\pp) \;\cong\; \Hom_{R/\pp}(\ccc, R/\pp) \;\cong\; \left(R/\pp :_{\Quot(R/\pp)} \ccc\right).
	$$
	Therefore, there exists a nonzero element ${d_{\bb/\aaa}} \in \Quot(R/\pp)$ such that $\overline{\delta}_{\bb/\aaa}(f) = \psi\left(\overline{f}\right) \cdot {d_{\bb/\aaa}}$  for all $f \in \bb$. 
	This completes the proof of the proposition.
\end{proof}

Another important part of our proof is the lemma below, where we utilize Huneke's uniform Artin-Rees lemma from \cite{HUNEKE_UNIFORM}.	
	
\begin{lemma}
	\label{lem_uniform_AR}
	Assume the notation and assumptions of  \autoref{prop_find_diff_op}.
	 There is an integer $c_{\bb/\aaa}$ only depending on $\aaa$ and $\bb$ such that, for any ideal $J$ in $R$ with image $I$ in $\Rred$ and any $n \ge 0$,  if $f \in \bb$ and $\delta_{\bb/\aaa}(f) \in I^{n+c_{\bb/\aaa}}$, then $f \in \bb\cap J^n + \aaa$.
\end{lemma}	
\begin{proof}
	Let $K \colonequals (J+\pp)/\pp \subset R/\pp$.
	Let $f \in \bb$ with $\delta_{\bb/\aaa}(f) \in  I^n$.
	This implies $\overline{\delta}_{\bb/\aaa}(f) =\psi(\overline{f}) {d_{\bb/\aaa}} \in K^{n}$.
	Write ${d_{\bb/\aaa}} = s/t$ with $s,t \in R/\pp$ and notice that $\psi(\overline{f})s \in K^{n}$.	
	From the uniform Artin-Rees lemma  \cite[Theorem 4.12]{HUNEKE_UNIFORM}, we obtain a constant $c_1$ such that 
	$$
	\psi(\overline{f}) s \;\in\; K^{n} \cap (s) \;\subset\; s K^{n-c_1}.
	$$
	Since $s$ is a nonzero element in the domain $R/\pp$, it follows that $\psi(\overline{f}) \in K^{n-c_1}$.
	Consider the ideal $\ccc \colonequals \IM(\psi) \subset R/\pp$.
	Again, by the uniform Artin-Rees lemma, there is a constant $c_2$ such that 
	$$
	\psi(\overline{f}) \;\in\; K^{n-c_1} \cap \ccc \;\subset\; \ccc K^{n-c_1-c_2}.
	$$
	Via $\psi : \bb/\aaa \xrightarrow{\;\cong\;} \ccc \subset R/\pp$, we obtain the induced isomorphism $J^{n-c_1-c_2} \cdot \bb/\aaa \xrightarrow{\;\cong\;} K^{n-c_1-c_2}\cdot \ccc$.
	It then follows $f \in \bb \cap J^{n-c_1-c_2} + \aaa$.
	So the claim of the lemma holds with $c_{\bb/\aaa}\colonequals c_1+c_2$.
\end{proof}

%\begin{theorem}\label{thm:main}
%	Assume \autoref{setup}.
%	Then, there is a set of  Noetherian operators $\delta_1,\ldots,\delta_m \in \Diff_{R/\kk}(R, \Rred)$ describing the zero ideal $(0) \subset R$ and an integer $c$ such that 
%	$$
%	I^{n+c} :_{\Rred} \lbrace\delta_1,\ldots,\delta_m\rbrace \;\subseteq\; J^n
%	$$ 
%	for any ideal $J$ in $R$ with image $I$ in $\Rred$ and any $n \ge 0$.
%\end{theorem}	

After the previous developments, we are ready for the proof of our main result.

\begin{proof}[Proof of \autoref{thm:main}]
	Choose a filtration $0 = \aaa_0 \,\subsetneqq\, \aaa_1 \,\subsetneqq\, \cdots \,\subsetneqq\, \aaa_k = R$ from \autoref{lem_filtration}.
	Let 
	$$
	\delta_i \;\colonequals \; \delta_{\aaa_i/\aaa_{i-1}} \;\in\; \Diff_{R/\kk}(R, \Rred)
	$$ 
	be the differential operator obtained by applying \autoref{prop_find_diff_op} to the ideals $\aaa_{i-1} \subsetneqq \aaa_{i}$.
	Let ${c_i \colonequals c_{\aaa_i/\aaa_{i-1}}}$ be the constant given by \autoref{lem_uniform_AR}.
	Notice that we may assume $\delta_k = \tau : R \rightarrow R/\aaa_{k-1}$ is the natural projection.
	 Let $e_1$ be the maximum of the orders of the operators $\delta_i$ and $e_2$ be the maximum of the constants $c_i$. 
	 Set $e = e_1+e_2$ and $c = ke$.
	
	Let $f \in R$ be an element such that $\delta_i(f) \in I^{n+c}$ for all $1 \le i \le k$.
	Since $\tau(f) = \delta_k(f) \in I^{n+c}$, we can write 
	$$
	f \;=\; f_{k-1} + f_k
	$$
	where $f_k \in J^{n+c} = \aaa_k \cap J^{n+ke}$ and $f_{k-1} \in \aaa_{k-1}$.
	Suppose by induction that we have obtained the equation 
	$$
	f \;=\; f_i + f_{i+1} + \cdots + f_k
	$$
	with $f_i \in \aaa_i$ and $f_j \in \aaa_j \cap J^{n+je}$ for all $j > i$.
	By applying the operator $\delta_i$, we get 
	$$
	\delta_i(f_i) \;=\; \delta_i(f) - \delta_i(f_{i+1}) - \cdots - \delta_i(f_k), 
	$$
	and since by \autoref{order-lemma}, $\delta_i(f) \in I^{n+ke}$ and $\delta_i(f_j) \in I^{n+je-e_1}$ for all $j > i$, it follows that $\delta_i(f_i)\in I^{n+(i+1)e-e_1}$.
	Therefore \autoref{lem_uniform_AR} yields $f_i = f_{i-1} + f_i'$ with $f_{i-1} \in \aaa_{i-1}$ and $f_i' \in \aaa_i \cap J^{n+ie}$.
	Consequently, by proceeding inductively we can write
	$$
	f \;=\; f_1 + f_2 + \cdots + f_k \quad \text{ with } \quad f_i \in \aaa_i \cap J^{n+ie}.
	$$
	In particular, we get the desired inclusion $f \in J^n$.
	If $\delta_1,\ldots,\delta_k$ do not give a set of Noetherian operators describing the zero ideal $(0)$, we can further adjoin differential operators from \autoref{cor_noeth_ops} until we get a set of Noetherian operators.
\end{proof}

\begin{remark}
	\label{rem_other_containment}
In the setting of \autoref{thm:main}, let $\delta_1,\dots,\delta_m \in \Diff_{R/\kk}(R,\Rred)$ be any set of Noetherian operators for $R$. 
If $e$ is the maximum of the orders of the operators $\delta_i$, then by \autoref{order-lemma} we have
\[ 
J^{n+e} \;\subseteq\; I^n :_{\Rred} \{\delta_1,\dots,\delta_m\}.
\] 
\end{remark}

Finally, we have the proofs of our main corollaries. 

\begin{proof}[Proof of \autoref{cor:Briancon_Skoda}] Fix a constant $c_1$ satisfying the conclusion of \autoref{thm:main}.
	By the uniform Brian\c{c}on-Skoda theorem \cite[Theorem 4.13]{HUNEKE_UNIFORM} applied to $\Rred$, there is a constant $c_2$ such that $\overline{I^{n+c_2}} \subseteq I^n$ independent of $I$. Take ${c=c_1+c_2}$. Then we have
\[	\overline{I^{n+c}} :_\Rred \lbrace\delta_1,\ldots,\delta_m\rbrace \ = \ \overline{I^{n+c_1+c_2}} :_\Rred \lbrace\delta_1,\ldots,\delta_m\rbrace \ \subseteq \ {I^{n+c_1}} :_\Rred \lbrace\delta_1,\ldots,\delta_m\rbrace \ \subseteq\ J^n.\qedhere\]
	\end{proof}

\begin{proof}[Proof of \autoref{cor:symb_powers}] 
Again, fix a constant $c_1$ satisfying the conclusion of \autoref{thm:main}.
By  \cite[Theorem~1.1]{HH_SYMB_POWERS} applied to $\Rred$, we get the containment $I^{(nd)} \subset I^n$ for any ideal $I \subset R$.
Take ${c=c_1d}$. 
Then we have
		\[	I^{(nd+c)} :_\Rred \lbrace\delta_1,\ldots,\delta_m\rbrace \ = \ I^{((n+c_1)d)} :_\Rred \lbrace\delta_1,\ldots,\delta_m\rbrace \ \subseteq \ {I^{n+c_1}} :_\Rred \lbrace\delta_1,\ldots,\delta_m\rbrace \ \subseteq\ J^n.\qedhere\]
\end{proof}

\section*{Acknowledgements} We thank the anonymous referee for their useful comments on this paper. This material is based upon work supported by the National Science Foundation under
Grant No. DMS-1928930 and by the Alfred P. Sloan Foundation under grant G-2021-16778,
while the authors were in residence at the Simons Laufer Mathematical Sciences
Institute (formerly MSRI) in Berkeley, California, during the Spring 2024 semester. 
The first author was partially supported by NSF grant DMS-2502321.
The second author was partially supported by NSF CAREER Award DMS-2044833. 
The second author also thanks SECIHTI/CONAHCYT, Mexico, for its support with grant CF-2023-G-33.

\end{document}